\documentclass[11pt,a4paper,leqno,twoside]{amsart}
\usepackage{amsmath,amssymb,amsthm,enumerate,hyperref,color}
\usepackage{hyperref}
\numberwithin{equation}{section}
\newtheorem{theorem}{Theorem}[section]

\newtheorem{proposition}[theorem]{Proposition}
\newtheorem{lemma}[theorem]{Lemma}

\newtheorem{remark}[theorem]{Remark}

\def\l{{\lambda}}

\def\s{{\sigma}}

\newcommand{\R}{\mathbb{R}}
\newcommand{\N}{\mathbb{N}}

\newcommand{\n}{\mathop{N}}

\newif\ifcomment \commentfalse
\def\commentON{\commenttrue}

\long\outer\def\BC#1\EC{\ifcomment \sloppy \par \# \ldots\dotfill
{\em #1} \dotfill \# \par \fi } \commentON

\newcommand{\remove}[1]{}

\begin{document}
\title[Nodal solutions..]{Nodal solutions for Lane-Emden problems in almost-annular domains}
\author[A.~L.~Amadori, F.~Gladiali, M.~Grossi]{Anna Lisa Amadori$^\dag$,  Francesca Gladiali$^\ddag$ and Massimo Grossi$^\sharp$}
\thanks{The authors are members of the Gruppo Nazionale per
l'Analisi Matematica, la Probabilit\'a e le loro Applicazioni (GNAMPA)
of the Istituto Nazionale di Alta Matematica (INdAM). The second author is supported by PRIN-2009-WRJ3W7 grant}
\date{\today}
\address{$\dag$ Dipartimento di Scienze Applicate, Universit\`a di Napoli ``Parthenope", Centro Direzionale di Napoli, Isola C4, 80143 Napoli, Italy. \texttt{annalisa.amadori@uniparthenope.it}}
\address{$\ddag$ Matematica e Fisica, Polcoming, Universit\`a di Sassari, via Piandanna 4, 07100 Sassari, Italy. \texttt{fgladiali@uniss.it}}
\address{$\sharp$ Dipartimento di Matematica, Universit\`a di Roma ``La Sapienza'',      \texttt{massimo.grossi@uniroma1.it}}
 
\begin{abstract}
In this paper we prove an existence result  to the problem
$$\left\{\begin{array}{ll}
-\Delta u = |u|^{p-1} u \qquad & \text{ in } \Omega , \\
u= 0 & \text{ on } \partial\Omega,
\end{array} \right.
$$
where $\Omega$ is a bounded domain in $\R^{\n}$ which is a perturbation of the annulus. Then there exists a sequence $p_1<p_2<..$ with $\lim\limits_{k\rightarrow+\infty}p_k=+\infty$ such that for any real number $p>1$ and $p\ne p_k$ there exist at least one solution with $m$ nodal zones.
\\
In doing so, we also investigate the radial nodal solution in an annulus: we provide an estimate of its Morse index and analyze the asymptotic behavior as $p\to 1$. 



\end{abstract}
\maketitle

\section{Introduction}

We are interested in the existence of nodal solutions to the Lane-Emden problem 
\begin{equation} \label{eq1}
\left\{\begin{array}{ll}
-\Delta u = |u|^{p-1} u \qquad & \text{ in } \Omega , \\
u= 0 & \text{ on } \partial\Omega,
\end{array} \right.
\end{equation} 
where $p>1$ and $\Omega$ is a smooth bounded domain in $\R^{\n}$, $\n\ge 2$. 

Addressing this problem in full generality is hard, and 
the answer changes according to the features (geometrical or topological) of the domain $\Omega$ and on the exponent $p$ of the nonlinear term.
A wide literature is available on this subject, and many interesting results have been obtained. 
For example, if $1<p<\frac{\n+2}{\n-2}$ when $\n\geq 3$ and for any $p$ when $N=2$, the compactness of the embedding of $H^1_0(\Omega)$ into $L^{p+1}(\Omega)$
gives the existence of {\em infinitely many} solutions, to \eqref{eq1} in any  smooth domain $\Omega$. On the other hand, when the exponent $p$ becomes critical or supercritical, 
 i.e.  $p\ge\frac{\n+2}{\n-2}$ for $\n\geq 3$, the compactness of the previous embedding can fail 
and so does in general the existence of solutions. 
Indeed the classical Pohozaev identity \cite{Po} implies that in this case, if $\Omega$ is starshaped with respect to one of its point, then \eqref{eq1} does not admit solutions.
The existence can be restored when
the domain $\Omega$ exhibits an hole. 
The simplest example is the case of the annulus where a radial solution always exists even if the exponent $p$ is supercritical.
We quote also the papers \cite{BC} and  \cite{C} where the existence of a positive solution is proved in the critical case  in a general domain whit holes.
If $p>\frac{\n+2}{\n-2}$ the existence of positive solutions  have been established in \cite{dPW} in domains with a small circular hole, while \cite{DW07} examines the case of nodal solutions. Both these papers rely on a perturbation argument around the exterior domain. Finally, we want to quote the paper \cite{BCGP} where the existence of a positive solution in an expanding annular type domain is proved if the radius of the domain is large enough (see also the references therein for other existence results). 
Here we focus onto  domains $\Omega$ which are small perturbations of an annulus, namely
\begin{equation}\label{domain}
  \begin{array}{c} \Omega_t = \{ x+ t \s(x) \, : \, x\in A\}, \\
    \mbox{where $A=\{x\in\R^{\n}\, : \, a<|x|<b\}$ is an annulus} \\ \mbox{and $\s:\bar A \to \R^{\n}$ is a smooth function.}
\end{array}\end{equation}
This perturbation has been used in \cite{DD} to study the Gelfand problem on a deformed ball, and later also by \cite{cowan14} for the H\'enon problem.

Here we prove existence of positive or nodal solutions to \eqref{eq1} if the exponent of the nonlinear term is different from a sequence of values that accumulates at $+\infty$. Our result does not depend on the measure of the annulus $A$ that can be small or large and does not depend on the shape of the hole and produces a nodal solution whose profile is close to the radial nodal solution in the annulus.
Some existence results in domains which are very general perturbations of a fixed domain $\Omega$ have been obtained in \cite{D} using the Leray-Schauder degree in the subcritical case.
Finally let us observe that it is one of the very few results for nodal solutions in the supercritical range.\\
Our main result is the following:

\begin{theorem}\label{teo1}
Let $m$ be a positive integer and $p$ a real number greater than $1$. Then there exists a sequence of exponents $1<p_1<p_2<\cdots < p_k \nearrow +\infty$ such that
for $p\ne p_k$ there exists a classical solution of \eqref{eq1} with $m$ nodal regions in $\Omega=\Omega_t$ for $t$ small enough.
\end{theorem}
In the case of a large annulus, i.e. $a=R$ and $B=R+1$ with $R$ large enough Theorem \ref{teo1} extends the existence result in \cite{BCGP} to a more general annular type expanding domain and to the case of nodal solutions. Let us stress that our proof looks like easier.\\
The previous result relies on a the use of the Implicit Function Theorem once we have studied the linearized problem associated to \eqref{eq1} when $\Omega$ is an annulus and $u$ is its  radial solutions (see Section \ref{s1} for the properties of this solution). Our main result in this direction, which in our opinion is interesting itself, is given by the following proposition,\par
\begin{proposition}[Characterization of degeneracy]\label{teo:annulus_nondegeneracy}
Let $\Omega$ be an annulus of $\R^{\n}$ with $\n\ge 2$ and $v_p$ a radial solution to \eqref{eq1} with $m$ nodal zones. Then $v_p$ is radially nondegenerate, and it is degenerate if  and only if 
\begin{align}\label{i3}
  \nu_l(p)=-j(\n-2+j), \quad && \mbox{ for some $l=1,\cdots m$ and $j\ge 2$}
  \intertext{or} \label{im}
  \nu_{m}(p)=-(\n-1) \qquad &&
\end{align}
where $\nu_l(p)$ is the $l^{th}$ eigenvalue of problem \eqref{eq:linauto}.
\end{proposition}
In some sense the previous proposition characterizes the ``bad values'' of $p$ where the Implicit Function Theorem does not hold and generalizes the study of the degeneracy of the radial positive solution in the paper \cite{GGPS} to the case of nodal solutions. Since it is not possible to solve the equations \eqref{i3}, \eqref{im} explicitly, we will derive that they have a countable number of solutions by studying them for $p$ close to $1$ and $+\infty$ and applying some ideas of  \cite{DW07}.
This allows to prove the following:
\begin{proposition}[Degeneracy points]\label{teo:annulus_degeneracy}
Let $\Omega$ be an annulus in $\R^{\n}$ with $\n\ge 2$ and $v_p$ a radial solution to \eqref{eq1} with $m$ nodal zones. Then there exists a sequence $1<p_1<p_2<\cdots < p_k \nearrow +\infty$ such that $v_{p}$ is  degenerate if and only if $p=p_k$. Moreover the Morse index of $v_p$ goes to $+\infty$ as $p\to \infty$.
\end{proposition}
Propositions \ref{teo:annulus_nondegeneracy} and \ref{teo:annulus_degeneracy} extend some properties of the radial solution $v_p$ studied in \cite{PS} to the case of nodal solutions with $m\geq 2$. The characterization of degeneracy in Proposition \ref{teo:annulus_nondegeneracy} is the key ingredient in \cite{GGPS} to prove the bifurcation of nonradial solutions from the positive radial solution in the annulus. Unfortunately in the case of nodal solutions some technical problems do not allow to conclude. We believe anyway that this problem deserves further study.
\par
Another interesting byproduct of Proposition \ref{teo:annulus_nondegeneracy} is an estimate from below for the Morse index.
\begin{proposition}[Morse index]\label{teo:morseindex}
  Let $\Omega$ be an annulus in $\R^{\n}$ with $\n\ge 2$, $p>1$ and $v_p$ a radial solution to \eqref{eq1} with $m$ nodal zones. Then its Morse index is strictly greater than $(m-1)(\n+1)$.
\end{proposition}
Such estimate improves the ones obtained in \cite[Theorem 1.1]{AP04} and \cite[Theorem 2.2]{BDG06} in the particular case of power nonlinearity.
\par
The paper is organized as follows: in Section  \ref{s1} we recall some properties of the radial solution to  \eqref{eq1} in the annulus, in  Section  \ref{s2} we study the degeneracy of the radial solution, we prove Proposition \eqref{teo:annulus_nondegeneracy} and we study the set of solutions of the equation \eqref{i3} obtaining Proposition \ref{teo:annulus_degeneracy} from a careful study of the asymptotic of the radial solution as $p\to 1$. Finally in Section  \ref{s3} we prove Theorem  \eqref{teo1} and some qualitative properties of the solution.

\section{Preliminaries on radial solutions in the annulus}\label{s1}
Let  $A=\{x\in\R^{\n}\, : \, a<|x|<b\}$ be an annulus and $\n\ge 2$.
We focus here on radial solutions to the problem 
\begin{equation}\label{eq2}
\left\{\begin{array}{ll}
-\Delta v = |v|^{p-1}v \qquad & \text{ in } A , \\
v= 0 & \text{ on } \partial A,
\end{array} \right.
\end{equation}
which have precisely $m$ nodal zones.
Since $v$ and $-v$ solve  \eqref{eq2} we fix the sign of the solution assuming that $v'(a)>0$.

For $m=1$, we are actually looking at positive solutions to
\begin{equation}\label{eq1A}
\left\{\begin{array}{ll}
-\Delta u = u^p \qquad & \text{ in } A , \\
u>0 & \text{ in } A , \\
u= 0 & \text{ on } \partial A.
\end{array} \right.
\end{equation} 
Problem  \eqref{eq1A} has an unique radial solution (see, for instance, \cite{NN}),  which  we denote by $u_p$.
It is radially nondegenerate for all $p$, and  nondegenerate for all $p$ except an increasing sequence  $1<p_1<p_2<\cdots < p_k \nearrow +\infty$ (see \cite[Lemma 2.3 and Section 4]{GGPS} for details).

For $m\ge 2$,  existence of  a solution for \eqref{eq2} comes from a standard application of the Nehari method. 
For $a\le \alpha< \beta\le b$, we write $A(\alpha,\beta)$ for the annulus with radii $\alpha$ ad $\beta$ and $H(\alpha,\beta)$ for $H^{1}_{0,r}(A(\alpha,\beta))$, the space of radial functions belonging to $H^1_0(A(\alpha,\beta))$.
 On every $H(\alpha,\beta)$, we may define the energy functional
\[{\mathcal E}(v)= \frac{1}{2}\int_{A(\alpha,\beta)} |\nabla v|^2-\frac{1}{p+1}\int_{A(\alpha,\beta)} |v|^{p+1},\]
and the set
\begin{align*}
& {\mathcal N}(\alpha,\beta)=\left\{ v\in H(\alpha,\beta) \, : \, \int_{A(\alpha,\beta)} |\nabla v|^2 = \int_{A(\alpha,\beta)} |v|^{p+1} \right\} 
. &
\end{align*}
It is well known that the nontrivial positive radial solution of the problem \eqref{eq2} in the annulus $A(\alpha,\beta)$ is a critical value of $\mathcal E$, that can be seen as a Mountain Pass point  on $H(\alpha,\beta)$ or as a minimum point on ${\mathcal N}(\alpha,\beta)$.
A nodal radial solution with exactly $m$ nodal zones and zeros $a=r_0<r_1<r_2<..<r_m=b$ can be produced by solving the minimization problem
\begin{equation}\label{neharimin}
\Lambda(r_1,\cdots r_{m-1}) = \min\left\{  \sum\limits_{i=1}^m \inf\limits_{{\mathcal N}(r_{i-1},r_{i})} {\mathcal E} \, : \, a=r_0<r_1<\cdots<r_m=b\right\}.
\end{equation}

\begin{theorem}[Existence and uniqueness of the radial solution]\label{teo:annulus_existence_scsol}
Let $p>1$ and $m$ be a positive integer. Problem \eqref{eq2} admits exactly one radially symmetric nodal solution $v_p=v_p(r)$ with precisely $m$ nodal zones and $v_p'(a)>0$. Moreover such solution realizes the minimum of \eqref{neharimin}.
\end{theorem}
We do not report the details of the existence part of the proof, which are very next to \cite[Theorem 2.1]{BW1993}, and somehow easier (see also Remark 2.2.a in the same paper). We also mention \cite{DW07}, where the same method is applied. Concerning uniqueness, it has been established in \cite[Theorem 3.1]{NN}

\begin{remark}\label{solposincoll}
Let $v_p$ be the radial solution of \eqref{eq2} and $a=r_0<r_1<\cdots<r_m=b$ its zeros. Then  $u_i(x)= (-1)^{i-1} \,v_p(x) \, 1_{\{r_{i-1}\le|x|\le r_{i}\}} \in {\mathcal N}(r_{i-1},r_i)$ is the only positive radial solution to \eqref{eq1A} in the annulus $A(r_{i-1},r_{i})$, as $i=1,\cdots,m$. We recall for future convenience that every $u_i$ is radially nondegenerate and its radial Morse index is 1.
\end{remark}

 \section{The linearization at $v_p$.}\label{s2}
In this section we  investigate the  nondegeneracy of $v_p$, precisely we want to characterize the values of $p$  such that the linearized problem
\begin{equation}\label{eq:lin}
\left\{\begin{array}{ll}
-\Delta w = p |v_p|^{p-1} w & \qquad \text{in } A, \\
w=0 & \qquad \text{ on } \partial A
\end{array}\right.
\end{equation} 
has a nontrivial solution.
As standard, we decompose any solution $w$ along $Y_k$, the space of the eigenfunctions of the Laplace-Beltrami operator on the sphere ${\mathbb S}^{\n-1}$, and write
\[w(x)=\sum\limits_{k=0}^{\infty} \phi_k(r) \, Y_k(\theta) , \quad a<r<b , \quad \theta\in {\mathbb S}^{\n-1}.\]
The components $\phi_k$ then satisfy the differential equations
\begin{align}\label{eq:linautorad}
\left\{\begin{array}{ll}
-\phi_k''- \dfrac{\n-1}{r} \phi_k'= \left(p |v_p|^{p-1} -\dfrac{\lambda_k}{r^2} \right)\phi_k & \qquad a<r<b , \\
\phi_k(a)=\phi_k(b)=0, &\end{array}\right.
\end{align}
where $\lambda_k$ is the eigenvalue associated to $Y_k$, i.e. $\lambda_k=j(\n-2+j)$ for some $j\in\N$.
\\
We also address to the one-dimensional  problem
\begin{align}\label{eq:linauto}
\left\{\begin{array}{ll}
-\phi''- \dfrac{\n-1}{r} \phi'= \left(p |v_p|^{p-1} +\dfrac{\nu}{r^2} \right)\phi \qquad a<r<b , \\
\phi(a)=\phi(b)=0. \end{array}\right.
\end{align}
The Sturm-Liouville theory guarantees that all the eigenvalues of \eqref{eq:linauto} are simple and that are characterized as min-max:
\begin{equation}\label{numinmax}\nu_l(p)=  \inf\limits_{\dim(V)= l} \max\limits_{\phi\in V} \dfrac{\int_a^b r^{\n-1}\left(|\phi'|^2-p|v_p|^{p-1}\phi^2\right)dr}{\int_a^b r^{\n-3}\phi^2dr} ,\end{equation}
 where $V$ runs through subspaces of $H^{1}_{0,r}(A)$.

\begin{proof}[Proof of Proposition \ref{teo:annulus_nondegeneracy}]

Comparing \eqref{eq:linauto} and \eqref{eq:linautorad}, it is clear that $v_p$  is radially degenerate only if $\nu_l(p)=-\lambda_0=0$ for some $l$, and degenerate if there exist $l$ and $k$ such that $\nu_l(p)=-\lambda_k$.
\\
By the min-max characterization \eqref{numinmax}, it is immediately seen that \eqref{eq:linauto} has at least $m$ negative eigenvalues, because the functions $u_i$ introduced in Remark \ref{solposincoll} have disjoint supports and they all satisfy
\begin{align*}
\int_a^b r^{\n-1}\left(|u_i'|^2-p|v_p|^{p-1}u_i^2\right)dr = &  \int_{r_{i-1}}^{r_{i}} r^{\n-1}\left(|u_i'|^2-p|u_i|^{p+1}\right)dr \\
= & -(p-1)\int_{r_{i-1}}^{r_{i}}  r^{\n-1}|u_i'|^2dr <0 .
\end{align*}
Next claim concerns the $(m+1)^{th}$ eigenvalue.\par
{\em Claim: the $(m+1)^{th}$ eigenvalue of \eqref{eq:linauto} is positive.}\par
To show this we look at  the auxiliary function $z=r v_p^{\prime}+\frac{2}{p-1}v_p$, which satisfies the equation in \eqref{eq:linauto} with $\nu=0$, but not the boundary condition. Let us prove that $z$ has exactly $m$ zeroes. Actually, as $v_p$ is the positive radial solution of \eqref{eq1}  in the annulus $A(r_{i-1},r_i)$ ($i=1,\dots m$), it follows that  $z(r_{i-1})=r_{i-1}v_p^{\prime}(r_{i-1})$ and $z(r_{i})=r_iv_p^{\prime}(r_i)$ are nonzero (otherwise $v_p$ and $v_p^{\prime}$ should vanish at the same point, implying $v_p \equiv 0$) and have opposite sign. Hence $z$ has at least one zero in any sub-interval,  and $z'\neq 0$ at any point where $z=0$ (otherwise also $z\equiv 0$).\par
Finally $z$ has not more than two nodal zones in any sub-interval $(r_{i-1},r_i)$  because otherwise it should be a sign-changing eigenfunction on a subdomain, contradicting the fact that $v_p$ has radial Morse index one in the annulus $A(r_{i-1},r_i)$.
\\
On the other hand, the $(m+1)^{th}$ eigenfunction of \eqref{eq:linauto} has $m+2$ zeroes in $[a,b]$, by the classical Sturm Liouville Theorem.
If the $(m+1)^{th}$ eigenvalue $\nu_{m+1}(p)$ where nonpositive, we could apply the Sturm-Picone Comparison Theorem and obtain that $z$ has 
at least $m+1$ zeros and this gives a contradiction proving the claim.
\par
In particular, this shows that the Morse index of problem \eqref{eq:linauto} is $m$ and $\nu_l(p)\neq 0$ for every $l$.
Then $v_p$ is radially nondegenerate, and the equality $\nu_l(p)=-j(\n-2+j)$ can hold only for $l\le m$ and $j\ge 1$.
Actually if $j=1$, the equality $\nu_l(p)=-(\n-1)$ can hold only for $l=m$, because $\nu_1(p)<\dots<\nu_{m-1}<-(\n-1)$ for all $p$.
To see this fact, we introduce another auxiliary function $\zeta:=v'_p$: it solves
    \[ -\zeta''-\dfrac{\n-1}{r}\zeta'=\left(p|v_p|^{p-1} -\dfrac{\n-1}{r^2}\right) \zeta \]
    and it has at least $m$ zeros inside $(a,b)$.
    Comparing this equation with \eqref{eq:linauto} by means of Sturm-Picone Comparison principle yields that, if $-(\n-1)\le \nu_l(p)$, then the related eigenfunction should have at least $m-1$ internal zeros, i.e. $l\ge m$.
\end{proof}

  The characterization of Proposition \ref{teo:annulus_nondegeneracy} allows to compute the Morse index of radial solutions, even though in a not completely explicit way. Anyway it suffices to give an estimate from below. We prove here Proposition \ref{teo:morseindex}.
\begin{proof}[Proof of Proposition \ref{teo:morseindex}]
  As explained in \cite[Lemmas 2.1 and 2.2]{GGPS}, the Morse index of the radial solution $v_p$ is exactly the sum of the dimensions of the eigenspace of the spherical harmonics (related to $j(N-2+j)$) such that $\nu_l(p)+j(N-2+j)<0$ for some $j\geq 1$ and for some $l=1,\dots,m$, i.e.
  \begin{equation}\label{morseformula}
    m(v_p)=\sum_{l=1}^m\sum _{j<J_l(p)}\frac{(N+2j-2)(N+j-3)!}{(N-2)! j!},
    \end{equation}
where $J_l(p)=\left(\sqrt{(\n-2)^2-4\nu_l(p)}-\n+2\right)/2$.
On the other hand in the proof of Proposition \ref{teo:annulus_nondegeneracy} it has been showed that $\nu_l(p)<-(\n-1)$ for $l=1,\dots m-1$ and $\nu_m(p)<0$. Hence $J_l(p)>1$ for $l=1,\dots m-1$ and $J_m(p)>0$, so that
\[
   m(v_p)\ge (m-1)\sum _{j=0}^1\frac{(N+2j-2)(N+j-3)!}{(N-2)! j!}+ 1 = (m-1)(\n+1)+1.
\]
\end{proof}

Next step stands in  showing that the equality $\nu_l(p)=-j(\n-2+j)$ is satisfied for a discrete increasing sequence of values of $p_k$. We shall deduce this fact by examining the behavior of the eigenvalues  $\nu_l(p)$ when $p$ approaches the ends of the existence range (i.e. $p\to+\infty$ and $p\to 1$) and then taking advantage of a sort of ``local analiticity'' of the map $p\mapsto\nu_l(p)$.
The asymptotic behavior of $v_p$ as $p\to+\infty$ has been deeply investigated in \cite{PS}. To our purpose it suffices to check that all the negative eigenvalues diverge.

\begin{lemma}\label{cpinfty}
  As $p\to+\infty$, it holds that $\nu_l(p)\to-\infty$ for $l=1,\dots m$. 
\end{lemma}

\begin{proof} 
By the min-max characterization of eigenvalues \eqref{numinmax}, we have that
\begin{align*}
\nu_1<\dots<\nu_{m}\le \max\left\{ \dfrac{\int_a^b r^{\n-1}\left(|\phi'|^2-p|v_p|^{p-1}\phi^2\right)dr}{\int_a^b r^{\n-3}\phi^2dr} \, : \, \phi = \sum\limits_{i=1}^{m}c_i u_i\right\}
\end{align*}
where $u_i\in {\mathcal N}(r_{i-1},r_i)$ are the positive solutions introduced in Remark \ref{solposincoll}. Now 
\begin{align*}
 \dfrac{\int_a^b r^{\n-1}\left(|\phi'|^2-p|v_p|^{p-1}\phi^2\right)dr}{\int_a^b r^{\n-3}\phi^2dr} 
= \dfrac{\sum\limits_{i=1}^{m}c_i^2 \int_{r_{i-1}}^{r_i} r^{\n-1}\left(|u_i'|^2-p u_i^{p+1}\right)dr}{\sum\limits_{i=1}^{m}c_i^2 \int_{r_{i-1}}^{r_i} r^{\n-3}u_i^{2}dr} \\
=(1-p)\dfrac{\sum\limits_{i=1}^{m}c_i^2 \int_{r_{i-1}}^{r_i} r^{\n-1}|u_i'|^2dr}{\sum\limits_{i=1}^{m}c_i^2 \int_{r_{i-1}}^{r_i} r^{\n-3}u_i^{2}dr} \le (1-p) {a^2} \lambda_1
\end{align*}
where $\lambda_1$ denotes the first eigenvalue of the Laplacian with zero Dirichlet boundary conditions on $A$. So the claim follows.
\end{proof}
Next we analyze the behavior of $v_p$ for $p$ close to $1$. The following result is in the spirit of \cite{Gro09}, where a detailed asymptotic picture is obtained in a more general framework, after assuming a-priori that $\|v_p\|_{2}^{p-1}$is bounded. Here we are able to prove that actually $\|v_p\|_{\infty}^{p-1}$ stays bounded, and then deduce that a suitable rescaling of $v_p$ converges to an eigenfunction of the Laplacian.
\begin{proposition}\label{p1}
Let $\l_{m}$ be the $m^{th}$ radial eigenvalue for the Laplacian in $A$ and $\psi_{m}$ be the corresponding radial eigenfunction. Then 
 \begin{equation}\label{a0}
    \| v_p\|_{\infty}^{p-1} \to \lambda_{m}\quad\text{ as } \, p\to 1 ,
    \end{equation}
    and
    \begin{equation}\label{a0bis}
    \frac{v_p} {\| v_p\|_{\infty}}\to \psi_{m}\quad \text{ in } C^2(A) , \quad \text{ as } \, p\to 1 .
\end{equation}
 \end{proposition}
\begin{proof}
  We first show that $\|v_p\|_{\infty}^{p-1}$ is bounded near $p=1$. 
  We assume by contradiction that there exists a sequence $p_n\to 1$ such that
\[ t_n= \|v_{p_n}\|_{\infty}^{\frac{p_n-1}{2}} \to \infty \quad \text{ as } \, n\to +\infty,  \]
and take $q_n\in(a,b)$ a maximum point for $|v_{p_n}(r)|$.
Up to an extracted sequence, $q_n$ converges to some $q_0\in[a,b]$.
Let us show that
  \begin{equation}\label{a1}
   t_n (b-q_n)\not\longrightarrow 0\quad\hbox{as }n\rightarrow+\infty.
  \end{equation}
To see this let us denote by  $r_{n}$ the last internal zero of $v_{p_n}$, and notice that $b-q_n > (b-r_n)/2$: this is obvious if the maximum point $q_n$ does not belong to the last nodal region, otherwise it follows by the Gidas, Ni, Nirenberg monotonicity property \cite[Theorem 2]{GNN}. So, in order to prove \eqref{a1} it suffices to check that $t_n(b-r_n)\not\longrightarrow0$.
\par
We thus look at the rescaled function $\tilde v_n(r)= \frac{1}{\|v_{p_n}\|} v_{p_n}(r_n+\frac{r-1}{t_n})$, which satisfies
\[\left\{\begin{array}{ll}
-\tilde v_n^{\prime\prime}=\dfrac{N-1}{t_nr_n\!+\!r\!-\!1} \tilde v_n^{\prime}  + \tilde v_n^{p_n} , & r\in I_n=(1,1+t_n(b-r_n)) , \\
0< \tilde v_n(r) \le  1 ,  & r\in I_n\\
\tilde v_n(1)=0=\tilde v_n(1+t_n(b-r_n)). & 
\end{array}\right.\]
Multiplying the equation by $\tilde v_{n}$ and integrating by parts gives
  \begin{align*}
    \int_{I_n}|\tilde v'_{n}|^2 dr = &\int_{I_n} \left(\dfrac{N-1}{t_nr_n\!+\!r\!-\!1} \tilde v_n \tilde v_n^{\prime} +  \tilde v_{n}^{p_n+1}\right) dr \\
    \le & \dfrac{N-1}{t_na } \left(\int_{I_n} |\tilde v_n^{\prime}|^2dr\right)^{\frac{1}{2}} \left(\int_{I_n} \tilde v_n^2dr\right)^{\frac{1}{2}} + \int_{I_n}\tilde v_{n}^{2} dr \\ \le &\left( \dfrac{N-1}{a} (b-r_n) + (t_n(b-r_n))^2\right) \int_{I_n} |\tilde v_n^{\prime}|^2dr
  \end{align*}
  by Poincar\'e inequality, which implies that \eqref{a1} holds.
  \\
 A similar argument will be used to show that 
   \begin{equation}\label{a2}
   t_n (q_n-a)\not\longrightarrow0\quad\hbox{as }n\rightarrow+\infty.
  \end{equation}
 If by contradiction \eqref{a2} does not hold, we must have that $q_n\to a$; then denoting by  $s_n$ the first internal zero of $v_{p_n}$ and reasoning as before yields that  $t_n (s_n-a)$ is bounded away from zero. If $q_n$ is not contained in the first nodal region, then $t_n(q_n-a)$ does not vanish. Otherwise, the same monotonicity argument applied to the Kelvin transform of $v_{p_n}$ yields that  
    \[q_n-a > a \left(\left(\frac{1+(s_n/a)^{2-N}}{2}\right)^{\frac{1}{2-N}} -1\right) .\]
Since we are assuming that $q_n\to a$, it follows that also $s_n\to a$. So, for large values of $n$, the right-hand side behaves like $(s_n-a)/2$ and therefore we conclude that \eqref{a2} holds.
   \par
  Next we introduce the auxiliary function
\[u_n(r)= \dfrac{1}{\|v_{p_n}\|_{\infty}} v_{p_n}\left(q_n+\dfrac{r}{t_n}\right) ,  \]
that satisfies 
\[\left\{\begin{array}{ll}
-u_n'' - \dfrac{N-1}{t_nq_n+r} u_n' = |u_n|^{p_n-1} u_n , & \mbox{ in } (\alpha_n,\beta_n) , \\
|u_n(r)|\le |u_n(0)|= 1 , \, u'_n(0)=0 , & \\
u_n(\alpha_n)=0 = u_n(\beta_n)).
\end{array}\right.\]
Here $\alpha_n=t_n(a-q_n)$ and $\beta_n=t_n(b-q_n)$.
By \eqref{a1} and \eqref{a2}, as $n$ goes to infinity, the set $(\alpha_n,\beta_n)$  goes to an unbounded interval $I$ containing $0$. To fix notations, we take $I=(\alpha_o, +\infty)$ with $\alpha_o<0$. Besides
\begin{align*}
  |u'_n(r)|=& \left|\int_0^r u''_n d\rho\right| = \left|\int_0^r\left(\frac{N-1}{t_nq_n+\rho} u_n' + |u_n|^{p_n-1} u_n\right) d\rho\right| \\
  \le & \int_0^r\frac{N-1}{t_nq_n+\rho} |u_n'| d\rho + r
  \intertext{if $r>0$, or}
   |u'_n(r)| \le & \int_r^0\frac{N-1}{t_nq_n+\rho} |u_n'| d\rho - r
\end{align*}
if $r<0$. So by Gronwall's Lemma we deduce that $|u_n'(r)| \le r \left(\frac{t_nq_n+r}{t_nq_n}\right)^{N-1}$ if $r>0$, or $|u_n'(r)| \le |r| \left(\frac{t_nq_n}{t_nq_n+r}\right)^{N-1}$ if $r<0$. In any case $|u_n'(r)|\le c |r|$ then $u_n$ converges (locally uniformly) to a function $u$ that satisfies
\[\left\{\begin{array}{ll}
-u''  =  u , & \mbox{ in } (\alpha_o, +\infty) , \\
|u(r)|\le |u(0)|= 1 , \, u'(0)=0 . & 
\end{array}\right.\]
This is not possible because $u(r)=\cos r$, which has an infinite number of nodal zones.
Eventually we have proved that $\|v_p\|^{p-1}_{\infty}$ is bounded. \par
Now we are in position to show that \eqref{a0} and \eqref{a0bis} hold.

Let $p_n$ be a sequence such that $p_n\rightarrow 1$ as $n\to +\infty$ and let $v_n:=v_{p_n}$. The function $\bar{v}_n=\frac {v_n}{\| v_n\|_{\infty}}$ satisfies
\begin{equation}\nonumber
\begin{cases}
-\Delta \bar{v}_n=\| v_n\|_{\infty}^{p_n-1}|\bar{v}_n|^{p_n-1}\bar{v}_n & \hbox{ in }A\\
\bar{v}_n=0 & \hbox{ on }\partial A.
\end{cases}
\end{equation}
This implies that $\| v_n\|_{\infty}^{p_n-1}$ can not go to zero otherwise
 $\bar{v}_n$ would converge uniformly to zero and this is a contradiction with $\| \bar {v}_n\|_{\infty}=1$. \\
Therefore, up to a subsequence, $\| v_n\|_{\infty}^{p_n-1}\rightarrow\l$ and $\bar{v}_n$ converges uniformly in $A$ to a function $\bar v$.
Let us show that
\begin{equation}
\left(|\bar{v}_n|^{p_n-1}-1\right)\bar{v}_n \to 0 .
\end{equation}
For any fixed $n$, we have $\left(|\bar{v}_n|^{p_n-1}-1\right)\bar{v}_n =0$ if $\bar{v}_n= 0$, otherwise
  \begin{align*}
    \left|\left(|\bar{v}_n|^{p_n-1} -1\right)\bar{v}_n\right|\le& (p_n-1)\left|\log|\bar{v}_n|\int_0^1|\bar{v}_n|^{1+t(p_n-1)} dt \right| \\  \le& c (p_n-1) |\bar{v}_n|^{1/2} \le   c (p_n-1) .
  \end{align*}
So obviously $\| v_n\|_{\infty}^{p_n-1}|\bar{v}_n|^{p_n-1}\bar{v}_n \to \lambda\bar v$ and $\bar v$  solves
\begin{equation}\nonumber
\begin{cases}
-\Delta \bar{v}=\l \bar{v} & \hbox{ in }A\\
\bar{v}=0 & \hbox{ on }\partial A.
\end{cases}
\end{equation}
Finally the limit eigenfunction $\bar v$ is radial and  has exactly $m$ nodal zones.
 Actually by \eqref{a1} it follows that the last internal zero $r_n$ satisfies  $b-r_n> c \|v_{n}\|_{\infty}^{-\frac{p_n-1}{2}}$ and therefore the last nodal zone can not collapse to a null set. Similarly, this can not happen  for all the nodal zones.
 Inside each zone, $\bar v_n$ is strictly positive (or negative) and converges uniformly to $\bar v$. Hence $\bar v$ cannot change sign and Hopf Lemma guarantees that no further zero can appear.
\end{proof}
Next we deduce some information about the asymptotic of the eigenvalues $\nu_l(p)$ as $p\to 1$.
\begin{lemma}\label{cp1} 
For $p$ near to 1 we have that
\begin{align}\label{nu11}
\nu_{l}(p)\hbox{ are bounded from below for any }l\ge1,
\\ \label{nu12}
\lim\limits_{p\to 1^+}\nu_{m}(p) = 0.
\end{align} 
  \end{lemma}
\begin{proof}
To check \eqref{nu11} it is enough to show that $\nu_{1}(p)$ is bounded from below as $p\to 1$. 
By definition
$$\nu_1(p)=  \inf\limits_{\phi\in H^{1}_{0,r}(A)} \dfrac{\int_a^b r^{\n-1}\left(|\phi'|^2-p|v_p|^{p-1}\phi^2\right)dr}{\int_a^b r^{\n-3}\phi^2dr} .$$
From Proposition \ref{p1} we have 
$$p|v_p|^{p-1}=p\|v_p\|_{\infty}^{p-1}|\bar{v}_p|^{p-1}\leq C$$
as $p\to 1$. Then, for any $\phi\in H^{1}_{0,r}(A)$ we have as $p\to 1$
\begin{align*}
  \int_a^b r^{\n-1}\left(|\phi'|^2-p|v_p|^{p-1}\phi^2\right)dr\geq \int_a^b r^{\n-1}\left(|\phi'|^2-C\phi^2\right)dr\\
  \geq
 -C\int_a^b r^{N-1}\phi^2dr \ge -c b^2\int_a^b r^{N-3}\phi^2dr , 
\end{align*}
so that
$$\dfrac{\int_a^b r^{\n-1}\left(|\phi'|^2-p|v_p|^{p-1}\phi^2\right)dr}{\int_a^b r^{\n-3}\phi^2dr}\geq -c b^2 
$$
which gives that $\nu_1(p)\geq  -cb^2 $.\par
Next, since we already know that $\nu_{m}(p)<0$ for all $p$,  it suffices to check that $\underline{\nu}=\liminf\limits_{p\to 1^+}\nu_{m}(p)=0$.
  To this end, let $p_n\to 1$ so that $\nu_n:=\nu_{m}(p_n)\to\underline\nu$ and  let $\phi_{n}$ be the $m$-eigenfunction for problem \eqref{eq:linauto} with $p=p_n$, normalized so that $\|\phi_n\|_{\infty}=1$.
  We compare the eigenfunction $\phi_n$ with $\bar v_n= v_{p_n}/\|v_{p_n}\|_{\infty}$, that satisfies
\begin{align*}
\left\{\begin{array}{ll}
-\bar v_n''- \dfrac{\n-1}{r} \bar v_n'= |v_{p_n}|^{{p_n}-1}\bar v_n & \qquad a<r<b , \\
\bar v_n(a)=\bar v_n(b)=0. &\end{array}\right.
\end{align*}
Assume by contradiction that $\underline\nu<0$. Remembering that $\|v_{p_n}\|_{\infty}^{p_n-1}$ is bounded by Proposition \ref{p1} we get
\begin{align*}
  p_n|v_{p_n}|^{{p_n}-1} +\dfrac{\nu_n}{r^2} - |v_{p_n}|^{{p_n}-1} \le (p_n-1) c +\dfrac{\nu_n}{b^2} \le 0 
\end{align*}
for $n$ large enough. On the other hand, $\phi_n$ and $\bar v_n$ have the same number of zeros, hence Sturm-Picone comparison theorem yields that $\phi_n=\pm\bar v_n$. In particular $\phi_n$ and $\bar v_n$ solve the same equation, that is
\[ \nu_n = (p_n-1) r^2 |v_{p_n}|^{{p_n}-1} .
\]
Passing to the limit as $n\to+\infty$, and using again the boundedness of $|v_{p_n}|^{p_n-1}$, we end up with the contradiction $\underline{\nu}=0$.
\end{proof}

\begin{lemma}\label{isolated} 
The map $p\mapsto \nu_l(p)$ is locally analytic and the set $\{\nu_l(p)=-j(\n-2+j)\}$ consists of only isolated points.
\end{lemma}
\begin{proof}
By Lemmas \ref{cpinfty} and \ref{cp1}, we have that, for any fixed integer $j\ge1$, if $p$ solves \eqref{i3} or \eqref{im} then $p$ belongs to a compact set in $(1,+\infty)$.  Then, arguing as in \cite[Lemma 3.3 part (c)]{DW07}, the claim follows.
\end{proof}
Eventually, putting  together the characterization of the degenerate $p$ obtained in Proposition \ref{teo:annulus_nondegeneracy} with the information collected in Lemmas \ref{cpinfty}, \ref{cp1} and \ref{isolated} we are able to conclude the proof of Proposition \ref{teo:annulus_degeneracy}.
\begin{proof}[Proof of Proposition \ref{teo:annulus_degeneracy}]
The values of $p$ such that $v_p$ is degenerate are given by the solution of the equations \eqref{i3} or \eqref{im}. 
By Lemmas  \ref{cpinfty} and \ref{cp1} the equation $\nu_m(p)=-j(\n-2+j)$ admits at least a solution for any $j\geq 1$.
So the values of $p$ such that $v_p$ is degenerate build up an infinite set, which consists of isolated points by Lemma \ref{isolated}.\\
In addition the Morse index of $v_p$ is given by the formula \eqref{morseformula}, and from Lemma  \ref{cpinfty} it is easy to see that $J_l(p)$ goes to infinity together with  $p$. Then $m(v_p)\to +\infty$ as $p\to +\infty$.
\end{proof}

\section{Existence of solutions in annular type domains}\label{s3}
Here we prove our perturbation theorem.
\begin{proof}[Proof of Theorem \ref{teo1}]
We start by introducing a change of variable which puts into relation problem \eqref{eq1}  in $\Omega_t$ with a problem in an annulus. For $A=\{x\in\R^{\n}\, : \, a<|x|<b\}$ let us consider a smooth function $\s:\bar A \to \R^{\n}$ and define
\begin{equation}
\Omega_t = \{ x+ t \s(x) \, : \, x\in A\}
\end{equation} 
Note that for $t$ small enough $\Omega_t$ is diffeomorphic to the annulus $A$.
Moreover there is another smooth function $\tilde\s$ such that $x=y+t\tilde\s(y)\in A$ for $y\in\Omega_t$ (at least for small values of $t$). An immediate computation shows that finding a solution $u(y)$ of \eqref{eq1} in $\Omega_t$ is equivalent to find a   solution $v$ of
\begin{equation}\label{eq2tA}
\left\{\begin{array}{ll}
-\Delta v -L_tv= |v|^{p-1}v \qquad & \text{ in } A , \\
v= 0 & \text{ on } \partial A,
\end{array} \right.
\end{equation} 
where $L_t$ is a linear operator
\[ L_t v =  t  \sum\limits_{i,k} \partial^2_{y_i y_i} \tilde\s_k \, \partial_{x_k}v + 2t \sum\limits_{i,k} \partial_{y_i} \tilde\s_k \, \partial^2_{x_i x_k}v + t^2 \sum\limits_{i,j,k} \partial_{y_j} \tilde\s_i \, \partial_{y_j} \tilde\s_k \,\partial^2_{x_i x_k}v .  
\]

Note that for $t=0$ problem \eqref{eq2tA} gives back problem \eqref{eq2} on the annulus (or \eqref{eq1A} for $m=1$).
Next we follow \cite{cowan14} and define a function $F:\R\times C^{2,\gamma}_0(\bar A) \to C^{0,\gamma}_0(\bar A)$ by
\[ F(t,v) = -\Delta v - L_t v - |v|^{p-1 }v.\]
It is easily seen that $F$ is a $C^1$ map verifying $F(0,v_p)=0$.
In order to apply the Implicit Function Theorem we examine $D_vF(0,v_p)$, the Fr\'echet derivative of $F$ with respect to $v\in C^{2,\gamma}_0(\bar A)$ computed at $(0,v_p)$. Its kernel is described by the solutions $w\in C^{2,\gamma}_0(\bar A)$ to the linearized problem 
\[ -\Delta w = p \, |v_p|^{p-1}\,  w. \] 
Hence the map $D_vF(0,v_p)$ has a bounded inverse for all $p$ such that the related radial solution $v_p$ is nondegenerate in $C^{2,\gamma}_0(\bar A)$.
For $m>1$, it follows by Lemmas \ref{teo:annulus_nondegeneracy} and  \ref{isolated}  that there is only an increasing sequence of isolated values of $p$ where $v_p$ is degenerate.
In the case $m=1$, the same property was already established in \cite{GGPS}.

So the  Implicit Function Theorem applies and there is a continuum  of functions $v_t\in C^{2,\gamma}_0(\bar A)$  such that $F(t, v_t)=0$ and the number of its nodal zones coincides with the ones of $v_0$, at least for $|t|$ small, because the map  $t\mapsto v_t$ is continuous on $C^{2,\gamma}_0(\bar A)$. Eventually  $u_t(y):=v_t(x)$ is a solution of \eqref{eq1} in $\Omega_t$ with exactly $m$  nodal zones.
\end{proof}
We end this section by proving some additional properties of the solution in the perturbed domain.
\begin{proposition}
Let us consider the solution $u_p$ of problem \eqref{eq1} in $\Omega_t$ given by Theorem \ref{teo1}. Then the Morse index of the solution $u_p$ satisfies
\begin{equation}
m(u_p)=m(v_p)
\end{equation}
where $v_p$ is the radial solution in the annulus. Finally,
\begin{equation}\label{a5}
\lim\limits_{p\rightarrow+\infty}m(u_p)=+\infty
\end{equation}
\end{proposition}
\begin{proof}
Using the map $\s$ we get that the Morse index of $u_p$ in $\Omega_t$ is the same of the corresponding function $v_{t,p}$ in $A$. Let us show that, for $t$ small, we have that
\begin{equation}
m(v_{t,p})=m(v_p)
\end{equation}
By contradiction suppose that we have that there exists a sequence $t_n\rightarrow0$ such that $m(v_{t_n,p})\ne m(v_p)$. Then, since the Morse index is an integer, we deduce that $\lim\limits_{n\rightarrow+\infty}m(v_{t_n,p})\ne m(v_p)$. On the other hand, since $v_{t_n,p}\rightarrow v_p$ in $C^2(A)$ as $n\rightarrow+\infty$ we get a contradiction.\par
Finally \eqref{a5} follows by Lemma \ref{cpinfty}.
\end{proof}
Our last result provides some information on the shape of the solution in the perturbed annulus $\Omega_t$ at least as $p$ is close to $1$ and $+\infty$.
\begin{proposition}
Let $u_p$ be the solution in $\Omega_t$ given by Theorem \ref{teo1}. Then, for any $\epsilon>0$\par
$i)$ there exist $p_0=p_0(\epsilon)$ and $t_0=t_0(\epsilon)$ such that for any $1<p<p_0$ and $|t|<t_0$ we have
\begin{equation}\label{b6}
\left|\left|u_p-\psi_m(y+t\tilde\s(y))\right|\right|_{C^0(\Omega_t)}<\epsilon
\end{equation}
where $\psi_m$ is the function appearing in Proposition \ref{p1},\par
$ii)$ there exist $p_0=p_0(\epsilon)$ and $t_0=t_0(\epsilon)$ such that for any $p>p_0$ and $|t|<t_0$ we have
\begin{equation}\label{b7}
\left|\left|u_p-\omega(y+t\tilde\s(y))\right|\right|_{C^0(\Omega_t)}<\epsilon
\end{equation}
where $\omega(x)$ is the radial function which appears in Theorem 1.1 in \cite{PS}.
\end{proposition}
\begin{proof}
We have that \eqref{b6} follows by Proposition \ref{p1} and Theorem \ref{teo1} and \eqref{b7} follows again by Theorem \ref{teo1} and  by the result in \cite{PS}.
\end{proof}

\end{document}